\documentclass[a4paper,12pt]{article}%
\usepackage[utf8]{inputenc}
\usepackage[T1]{fontenc}
\usepackage{times}
\usepackage[bitstream-charter]{mathdesign}
\usepackage[all,2cell]{xypic}\UseAllTwocells
\usepackage{graphicx}
\usepackage{enumerate}
\RequirePackage{amsopn} 
\RequirePackage{amsbsy} 
\RequirePackage{amsmath} 
\RequirePackage{relsize} 
\RequirePackage{latexsym} 
\RequirePackage{theorem}%
\RequirePackage{thc}%
\newcounter{mysubsection}[section]%
\newcounter{mysubsubsection}[mysubsection]%
\theorembodyfont{\slshape}%
\newtheorem{corollary}[mysubsection]{Corollary}%
\newtheorem{lemma}[mysubsection]{Lemma}%
\newtheorem{proposition}[mysubsection]{Proposition}%
\newtheorem{question}[mysubsection]{Question}%
\newtheorem{theorem}[mysubsection]{Theorem}%
\theorembodyfont{\upshape}%
\newtheorem{example}[mysubsection]{Example}%
\newtheorem{remark}[mysubsection]{Remark}%
\def\qed{{\unskip\nobreak\hfil\penalty50%
  \hskip2em\hbox{}\nobreak\hfil$\square$%
  \parfillskip=0pt\finalhyphendemerits=0\par}}
\newenvironment{proof}%
  {\par\addvspace{\medskipamount}%
    \upshape%
    {\slshape\scshape
    Proof\hskip\labelsep}}%
  {\qed%
    \addvspace{\medskipamount}}%
\newcommand\Ann{\textrm{Ann}}
\newcommand\End{\textrm{End}}
\newcommand\Hom{\textrm{Hom}}
\newcommand\id{\textrm{id}}
\newcommand\ideal[1]{\mathfrak{\lowercase{#1}}}
\newcommand\Ideal[1]{\mathfrak{\uppercase{#1}}}
\newcommand\Ima{\textrm{Im}}
\newcommand\Ker{\textrm{Ker}}
\newcommand\rMod{\textbf{Mod}-}
\newcommand\sepa{\vspace*{-3mm}\setlength{\itemsep}{-2mm}}
\newcommand\Soc{\textrm{Soc}}
\newcommand\Supp{\textrm{Supp}}
\newcommand\blfootnote[1]{%
  \begingroup
  \renewcommand\thefootnote{}\footnote{#1}%
  \addtocounter{footnote}{-1}%
  \endgroup}

\begin{document}
\title{Lattice decomposition of modules\blfootnote{\today \newline
J. M. Garc\'{\i}a
    (\emph{jgarciah@ugr.es})
    Department of Applied Mathematics. University of Granada. E--18071 Granada.
    Spain.\newline
P. Jara (Corresponding author)
    (\emph{pjara@ugr.es})
    Department of Algebra and IEMath--GR (Instituto de Matem\'{a}ticas). University of Granada. E--18071 Granada.
    Spain.\newline
L. M. Merino
    (\emph{lmerino@ugr.es})
    Department of Algebra and IEMath--GR (Instituto de Matem\'{a}ticas). University of Granada. E--18071 Granada.
    Spain.\newline
2020 \emph{Mathematics Subject Classification:} 16D70, 18E10, 18E35, 13C60\newline
\emph{Key words:} module, ring, lattice, lattice decomposition, Grothendieck category.}}
\author{Josefa M. Garc\'{\i}a\\
        Pascual Jara\\
        Luis M. Merino}
\date{}
\maketitle

\begin{abstract}
The first aim of this work is to characterize when the lattice of all submodules of a module is a direct product of two lattices. In particular, which decompositions of a module $M$ produce these decompositions: the \emph{lattice decompositions}. In a first \textit{\'{e}tage} this can be done using endomorphisms of $M$, which produce a decomposition of the ring $\End_R(M)$ as a product of rings, i.e., they are central idempotent endomorphisms. But since not every central idempotent endomorphism produces a lattice decomposition, the classical theory is not of application. In a second step we characterize when a particular module $M$ has a lattice decomposition; this can be done, in the commutative case in a simple way using the support, $\Supp(M)$, of $M$; but, in general, it is not so easy. Once we know when a module decomposes, we look for characterizing its decompositions. We show that a good framework for this study, and its generalizations, could be provided by the category $\sigma[M]$, the smallest Grothendieck subcategory of $\rMod{R}$ containing $M$.
\end{abstract}

\section*{Introduction}

Let $M$ be a (unitary) right $R$--module over a (unitary) ring $R$; it is well known that decompositions of $M$, as a direct sum of two submodules, are parameterized by idempotent endomorphisms in $S=\End_R(M)$. Thus, if $M=N\oplus{H}$, there exists $e\in\End_R(M)$ such that $N=e(M)$, and $H=(1-e)(M)$. In general, $e$ is not necessarily central in $S$, hence it does not produce a decomposition of $S$ in a direct product of two rings. In this paper we deal with some special decompositions of modules so that the lattice $\mathcal{L}(M)$, of all submodules of $M$, will be a direct product of two lattices: the \emph{lattice decomposition} of $M$.

This kind of decompositions are of interest as if $M=N\oplus{H}$ is a lattice decomposition, then every submodule $X\subseteq{M}$ can be expressed as a direct sum, $X=(X\cap{N})\oplus(X\cap{H})$, and this property has great importance in order to study the structure of $M$.

From this point of view, we first recall that a lattice decomposition defines an element $N\in\mathcal{L}(M)$ which is distributive; this means that for any $X,Y\in\mathcal{L}(M)$, the sublattice of $\mathcal{L}(M)$ generated by $N,X$ and $Y$ is distributive: a well known notion in lattice theory; in addition, it is complemented, as $M=N\oplus{H}$. Elements of this kind have good properties as members of the lattice $\mathcal{L}(M)$.

We exploit the existence of complemented distributive submodules of a right $R$--module $M$. In order to characterize them first we deal with endomorphisms. Thus we show a characterization of those central idempotent endomorphisms in $\End_R(M)$ that define lattice decompositions of $M$. In the particular case of modules over a commutative ring, these idempotent endomorphisms are those that belong to the closure of $R$ in $\End_R(M)$, with respect to the finite topology. In addition, we show that every complemented distributive submodule $N\subseteq{M}$ is stable under any endomorphism $f\in\End_R(M)$.

The behaviour of complemented distributive submodules is also studied, thus it is shown that for any complemented submodule $N\subseteq{M}$, and any index set $I$ we get a complemented distributive submodule $N^{(I)}\subseteq{M^{(I)}}$. This property, together with the known characterization of distributive submodules as those submodules $N\subseteq{M}$ such that for every submodule $H\subseteq{M}$ the factor modules $N/(N\cap{H})$ and $H/(N\cap{H})$ have no non--zero isomorphic subfactors allow us to extend the theory to categories which are defined directly from $M$, as the category $\sigma[M]$. Indeed, we recover a decomposition theory for these categories showing that there exists a closed relationship between decomposition of $\sigma[M]$, as a product of two subcategories, and lattice decompositions of $M$ as a right $R$--module.

\medskip

The paper is organized in sections. In the first section we recall the notions of product of lattices and the consequences of the existence of a lattice decomposition. In sections two and three we study distributive submodules of a right $R$--module, and show that simple subfactors are decisive to characterize complemented distributive submodules. In particular, if the base ring $R$ is commutative, a direct sum decomposition $M=N\oplus{H}$ is a lattice decomposition if, and only if, $\Supp(N)\cap\Supp(H)=\varnothing$. One of the main aims is to relate complemented distributive submodules $N\subseteq{M}$ and central idempotent endomorphisms. Thus we show, in examples, that not every such idempotent endomorphism defines a distributive submodule, and in Theorem~\eqref{th:151011} we show that it is necessary and sufficient that this endomorphism stabilizes every submodule. In consequence, these idempotent endomorphisms are close to multiplication by elements of $R$, and in Proposition~\eqref{pr:151011} we show that they must belong to the closure of $R$ in $\End_R(M)$ with respect to the finite topology, whenever $R$ is commutative. In order to extend these results to categories, in section four we establish, Proposition~\eqref{pr:10}, showing that direct powers preserve complemented distributive submodules. The strong relationship between decompositions of the category $\sigma[M]$ and lattice decompositions of the module $M$ is also studied.

\medskip

References to undefined terms can be find either in the following papers:
\cite{Jensen:1966},
\cite{Rajaee:2013} and
\cite{Tuganbaev:2008},
or in the books:
\cite{GRATZER:1978},
\cite{STENSTROM:1975} and
\cite{WISBAUER:1988}.

\section{Product of lattices}

Let $L_1,L_2$ be lattices, and define in the cartesian product $L_1\times{L_2}$ the operations:
\[
\begin{array}{ll}
(a_1,a_2)\wedge(b_1,b_2)=(a_1\wedge{b_1},a_2\wedge{b_2})\textrm{ and}\\
(a_1,a_2)\vee(b_1,b_2)=(a_1\vee{b_1},a_2\vee{b_2}),
\end{array}
\]
then $(L_1\times{L_2},\wedge,\vee)$ is a lattice, and the canonical projections $p_i:L_1\times{L_2}\longrightarrow{L_i}$, $i=1,2$, are lattice maps. In addition, $(L_1\times{L_2},\{p_1,p_2\})$ is the product of $L_1$ and $L_2$ in the category of lattices and lattice maps.

Examples of lattices appear in many different contexts; we are interested in those lattices that appear in module theory, i.e., if $R$ is a (unitary) ring and $M$ a right (unitary) $R$--module, in the lattice $\mathcal{L}(M)$ of all submodules of $M$, and in the particular problem of characterizing when $\mathcal{L}(M)$ is the product of two lattices.

For any right $R$--module $M$ the lattice $\mathcal{L}(M)$ has extra properties in addition to those that define a lattice, for instance:
\begin{enumerate}[(1)]\sepa
\item
$\mathcal{L}(M)$ is \textbf{bounded}, i.e., there exists a bottom element, $0\subseteq{M}$, and a top one, $M$.
\item
$\mathcal{L}(M)$ is \textbf{modular}, i.e., for any $N_1,N_2,N_3\subseteq{M}$ such that $N_1\subseteq{N_3}$, we have $(N_1+N_2)\cap{N_3}=N_1+(N_2\cap{N_3})$.
\end{enumerate}

If $\mathcal{L}(M)$ is a product of lattices, using that $\mathcal{L}(M)$ is bounded, the following easy results holds.

\begin{lemma}\label{le:20200802}
If $\mathcal{L}(M)$ is the product of two lattices, say $\mathcal{L}(M)=\mathcal{L}_1\times\mathcal{L}_2$, with projections $\{p_1,p_2\}$, $M=(M_1,M_2)$, and $0=(0_1,0_2)$, then $\mathcal{L}_1$ satisfies the following properties:
\begin{enumerate}[(1)]\sepa
\item\label{it:le:20200802-1}
$\mathcal{L}_1$ is a bounded lattice with bottom $0_1$ and top $M_1$.
\item\label{it:le:20200802-2}
The map $q_1:\mathcal{L}_1\longrightarrow\mathcal{L}(M)$, defined $q_1(X)=(X,0_2)$, is a one--to--one lattice map.
\item\label{it:le:20200802-3}
There is a lattice isomorphism between $\mathcal{L}_1$ and $\{H\subseteq{M}\mid\;H\subseteq{M_1}\}$.
\item\label{it:le:20200802-4}
The map $h_1:\mathcal{L}_1\longrightarrow\mathcal{L}(M)$, defined $h_1(X)=(X,M_2)$, is a one--to--one lattice map.
\item\label{it:le:20200802-5}
There is a lattice isomorphism between $\mathcal{L}_1$ and $\{L\subseteq{M}\mid\;M_2\subseteq{L}\}$.
\end{enumerate}
The same properties hold for the lattice $\mathcal{L}_2$. In particular, $M$ is the direct product of $M_1$ and $M_2$, and there is a lattice isomorphism $\mathcal{L}(M)\cong\mathcal{L}(M_1)\times\mathcal{L}(M_2)$.
\end{lemma}
\begin{proof}
\eqref{it:le:20200802-1}. %
For any $X\in\mathcal{L}_1$ we have
$(0_1,0_2)=0\leq(X,0_2)$, hence $0_1=p_1(0_1,0_2)\leq{p_1(X,0_2)}=X$, and
$(X,M_2)\leq{M=(M_1,M_2)}$; therefore, $X=p_1(X,M_2)\leq{p_1(M_1,M_2)}=M_1$.
\par
\eqref{it:le:20200802-2} and \eqref{it:le:20200802-3}. %
It is clear that $q_1$ is a lattice map, and the announced isomorphism is given by $q_1$.
\par
\eqref{it:le:20200802-4} and \eqref{it:le:20200802-5}. %
It is clear that $h_1$ is a lattice map, and the announced isomorphism is given by $h_1$.
\par
Observe that $M_1\cap{M_2}=0$ and $M_1+M_2=M$, hence $M=M_1\times{M_2}$.
\end{proof}

A right $R$--module $M$ has a \textbf{lattice decomposition} whenever $\mathcal{L}(M)$ is a product of two nontrivial lattices.

It is clear that not every decomposition of a module $M$ as a direct product gives a lattice decomposition of $\mathcal{L}(M)$ in a product of lattices. See the following example.

\begin{example}
Consider the abelian group $M=\mathbb{Z}_2\times\mathbb{Z}_2$, the lattice of subgroups of $M$ is not a product of two nontrivial lattices; in particular, $\mathcal{L}(M)$ is not the product $\mathcal{L}(\mathbb{Z}_2)\times\mathcal{L}(\mathbb{Z}_2)$.
\end{example}

\begin{example}
A commutative ring $A$ has a lattice decomposition if, and only if, $A$ is the product of two nontrivial ideals. Indeed, if $\mathcal{L}(A)=\mathcal{L}_1\times\mathcal{L}_2$, there exist ideals $\ideal{a}_1,\ideal{a}_2\subseteq{A}$ such that $A=\ideal{a}_1\times\ideal{a}_2$.
Otherwise, if $A=\ideal{a}_1\times\ideal{a}_2$, there are idempotent elements $a_i\in\ideal{a}_i$, $i=1,2$, such that $1=a_1+a_2$. For any ideal $\ideal{a}\subseteq{A}$ we have $\ideal{a}=\ideal{a}a_1\times\ideal{a}a_2$, and an isomorphism $\mathcal{L}(A)\cong\mathcal{L}(\ideal{a}_1)\times\mathcal{L}(\ideal{a}_2)$.
\end{example}

\begin{example}
This result for non--commutative rings does not hold. Let us consider a field $K$ and the matrix ring $M_2(K)$ of all square matrices of order 2. The ideals $\ideal{a}_1=\begin{pmatrix}K&K\\0&0\end{pmatrix}$ and $\ideal{a}_2=\begin{pmatrix}0&0\\K&K\end{pmatrix}$ satisfy $M_2(K)=\ideal{a}_1\oplus\ideal{a}_2$. Otherwise, each $\ideal{a}_i$ is a simple right $M_2(K)$--module, hence $\mathcal{L}(\ideal{a}_i)=\{0,\ideal{a}_i\}$, but $\mathcal{L}(M_2(K))$ is not the product $\mathcal{L}(\ideal{a}_1)\times\mathcal{L}(\ideal{a}_2)$ because, for any $0\neq{a}\in{K}$, the right ideal $\begin{pmatrix}1&0\\a&0\end{pmatrix}M_2(K)$ is not in this product. See Corollary~\eqref{co:20200910} to determine when a ring $R$ have a lattice decomposition as right $R$--module.
\end{example}

Our aim in the next section shall be to show some characterizations of modules having a lattice decomposition.

\section{Distributive submodules}

Let $M$ be a right $R$--module. If $N\subseteq{M}$ is a submodule, there exists a short exact sequence $0\to{N}\to{M}\to{M/N}\to0$, and maps
\[
\xy
 (0,0)*+{\mathcal{L}(N)}="a",
 (30,0)*+{\mathcal{L}(M)}="b",
 (60,0)*+{\mathcal{L}(M/N)}="c",
 \ar@{->}"a";"b"^{i_*}
 \ar@/^+5ex/"b";"c"^{p_*}
 \ar@{->}"c";"b"_{p^*}
 \ar@/^+5ex/"b";"a"_{i^*}
\endxy
\]
Defined by:
\begin{enumerate}[(1)]\sepa
\item\label{it:20200803-1}
$i_*(X)=X$, for every $X\subseteq{N}$; it is a lattice homomorphism.
\item\label{it:20200803-2}
$i^*(Y)=Y\cap{N}$, for every $Y\subseteq{M}$; it satisfies $i^*(Y_1\wedge{Y_2})=i^*(Y_1)\wedge{i^*(Y_2)}$, but it is not a lattice homomorphism unless $N$ satisfies $(Y_1+Y_2)\cap{N}=(Y_1\cap{N})+(Y_2\cap{N})$ for any $Y_1,Y_2\subseteq{M}$.
\item\label{it:20200803-3}
$p_*(Y)=(Y+N)/N$, for every $Y\subseteq{M}$; it satisfies $p_*(Y_1\vee{Y_2})=p_*(Y_1)\vee{p_*(Y_2)}$, but it is not a lattice homomorphism unless $N$ satisfies  $(Y_1\cap{Y_2})+N=(Y_1+N)\cap(Y_2+N)$ for any $Y_1,Y_2\subseteq{M}$.
\item\label{it:20200803-4}
$p^*(Y/N)=Y$, for every $Y/N\subseteq{M/N}$; it is a lattice homomorphism.
\end{enumerate}
Thus, in the above diagram all maps are lattice maps if, and only if, $N$ satisfies conditions in \eqref{it:20200803-2} and in \eqref{it:20200803-3}. In \cite{GRATZER:1978} an element in a lattice satisfying  property in \eqref{it:20200803-3} is called a \textbf{distributive element}, and if it satisfies property in \eqref{it:20200803-2}, a \textbf{dual distributive element}, proving in \cite[Theorem III.2.6]{GRATZER:1978} that an element in a modular lattice is distributive if and only if it is dual distributive if, and only if, the sublattice generated by $N,Y_1$ and $Y_2$, in the former notation, is distributive.

We call a submodule $N\subseteq{M}$ \textbf{distributive} whenever $(Y_1\cap{Y_2})+N=(Y_1+N)\cap(Y_2+N)$ for any $Y_1,Y_2\subseteq{M}$, or equivalently if $(Y_1+Y_2)\cap{N}=(Y_1\cap{N})+(Y_2\cap{N})$ for any $Y_1,Y_2\subseteq{M}$, and observe that this is also equivalent to the condition that the sublattice of $\mathcal{L}(M)$, generated by $N,Y_1$ and $Y_2$, is distributive.

Our aim in this section is to characterize distributive submodules of a module. To do that we need  the following definition. Let $M$ be a right $R$--module, a \textbf{subfactor} of $M$ is a submodule of a homomorphic image of $M$. Observe that for any subfactor $L$ of a right $R$--module $M$, and any submodule $K\subseteq{L}$ we may build a commutative diagram
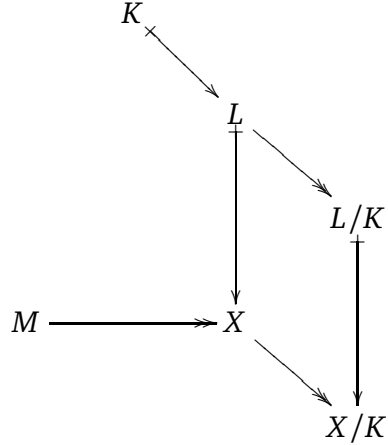
\[
\begin{xy}
\xymatrix{
&K\ar@{+->}[rd]&&\\
&&L\ar@{+->}[dd]\ar@{->>}[rd]\\
&&&L/K\ar@{+->}[dd]\\
M\ar@{->>}[rr]&&X\ar@{->>}[rd]&\\
&&&X/K
}\end{xy}
\]
Therefore, if $L$ is a subfactor of $M$, then $K$ and $L/K$ are also subfactors of $M$.

This situation can be enhanced if we make use of elements of the module $M$. So, distributive submodules can be also characterized in the following way; where we refer to \cite[Theorem 1.6]{Stephenson:1974} or \cite[Proposition 1.1]{Barnard:1981b} for condition (b), and to \cite{Davison:1974} for condition (c).

\begin{proposition}\label{pr:42}
Let $M$ be a right $R$--module, and $N\subseteq{M}$ be a submodule, the following statements are equivalent:
\begin{enumerate}[(a)]\sepa
\item
$N\subseteq{M}$ is a distributive submodule.
\item
$(N:m)+(mR:n)=R$ for any $m\in{M}$ and $n\in{N}$.
\item
For every submodule $H\subseteq{M}$ the modules $N/(N\cap{H})$ and $H/(N\cap{H})$ have no non--zero isomorphic subfactors.
\item
For every submodule $H\subseteq{M}$, the modules $N/(N\cap{H})$ and $H/(N\cap{H})$ have no simple isomorphic subfactors.
\item
For any $m\in{M}$ and $n\in{N}$, the cyclic modules $(n+(N\cap{mR}))R$ and $mR/(N\cap{mR})$ have no non--zero isomorphic subfactors.
\item
For any $m\in{M}$ and $n\in{N}$, the cyclic modules $(n+(N\cap{mR}))R$ and $mR/(N\cap{mR})$ have no simple isomorphic subfactors.
\end{enumerate}
\end{proposition}
\begin{proof}
(a) $\Rightarrow$ (b). %
By hypothesis we have $N\cap((n-m)R+mR)=(N\cap{(n-m)R})+(N\cap{mR})$; hence
\[
n=x+y,\mbox{ where }x\in{N\cap{(n-m)R}}\mbox{ and }y\in{N\cap{mR}}.
\]
Let $a,b\in{R}$ such that $x=(n-m)a$, hence $ma=na-x\in{N}$, and $y=mb$. In addition, we have $n(1-a)=n-na=x+y-na=(n-m)a+mb-na=m(b-a)\in{mR}$. As a consequence, $(N:m)+(mR:n)=R$.
\par (b) $\Rightarrow$ (a). %
For any $X,Y\subseteq{M}$ we always have $(N\cap{X})+(N\cap{Y})\subseteq{N\cap(X+Y)}$. On the other hand, let $n=x+y\in{N\cap(X+Y)}$, where $x\in{X}$ and $y\in{Y}$, and consider the pair $n\in{N}$ and $x\in{M}$. By hypothesis, we have $(N:x)+(xR:n)=R$, there exist $a\in(N:x)$, $b\in(xR:n)$ such that $a+b=1$, and we have: $n=xa+ya+nb$; since $xa,nb\in{N\cap{X}}$, hence $ya=n-xa-nb\in{N}$, whence $ya\in{N\cap{Y}}$. Therefore, $n\in(N\cap{X})+(N\cap{Y})$.
\par (b) $\Rightarrow$ (c). %
For any non--zero subfactor $SF_1$ of $N/N\cap{H}$, and any subfactor $SF_2$ of $H/(N\cap{H})$, let us consider the diagram
\[
\resizebox{14.5cm}{!}{
\begin{xy}
\xymatrix{
X\ar@{->>}[rr]\ar@{+->}[d]&&X/(N\cap{H})\ar@{+->}[d]\ar@{->>}[rr]^f&&SF_1\ar@{+->}[d]\ar[r]^\eta_\cong
 &SF_2\ar@{+->}[d]&&Y/(N\cap{H})\ar@{->>}[ll]_g\ar@{+->}[d]&&Y\ar@{+->}[d]\ar@{->>}[ll]\\
N\ar@{->>}[rr]&&N/(N\cap{H})\ar@{->>}[rr]&&\bullet
 &\bullet&&H/(N\cap{H})\ar@{->>}[ll]&&H\ar@{->>}[ll]
}\end{xy}}
\]
there exists $0\neq{x}\in{X}$ such that $f(\overline{x})\neq0$, where $\overline{x}=x+(N\cap{H})$. Let $y\in{Y}$ such that $\eta{f}(\overline{x})=g(\overline{y})$, where $\overline{y}=y+(N\cap{H})$. By the hypothesis $(N:y)+(yR:x)=R$; let $1=a+b$ with $a\in(N:y)$ and $b\in(yR:x)$, then $x=x(a+b)=xa+xb$. On the other hand, $\eta{f}(\overline{xa})=g(\overline{y}a)=0$, whence $\overline{xa}\in\Ker(f)$; since $\overline{xb}\in\Ker(f)$, we have $\overline{x}\in\Ker(f)$, which is a contradiction.
\par (c) $\Rightarrow$ (d), (e) $\Rightarrow$ (f). %
They are trivial.
\par (f) $\Rightarrow$ (b). %
Let $x\in{M}$ and $n\in{N}$, if $(N:x)+(xR:n)\neq{R}$, there exists a maximal right ideal $\ideal{m}\subseteq{R}$ such that $(N:x)+(xR:n)\subseteq\ideal{m}$, and for any $a\in\ideal{m}$ we have $1-a\notin\ideal{m}$, hence $1-a\notin(N:x),(xR,n)$. We proceed as follows:
\begin{enumerate}[(1)]\sepa
\item
Since $1-a\notin(N:x)$, then $(1-a)x\notin{N\cap{xR}}$, and for any $a\in\ideal{m}$ we have $\overline{x}\neq{\overline{x}a}$ in $M/(N\cap{xR})$, i. e., $\overline{x}\ideal{m}\subsetneqq{\overline{x}R}$, and $xR/(N\cap{xR})$ has a simple subfactor $\overline{x}R/\overline{x}\ideal{m}\cong{R/\ideal{m}}$.
\item
Since $1-a\notin(xR:n)$, then $n(1-a)\notin{N\cap{xR}}$, and for any $a\in\ideal{m}$ we have $\overline{n}\neq{\overline{n}a}$ in $M/(N\cap{xR})$, i. e., $\overline{n}\ideal{m}\subsetneqq{\overline{n}R}$, and $\overline{n}R$ has a simple subfactor $\overline{n}R/\overline{n}\ideal{m}\cong{R/\ideal{m}}$.
\end{enumerate}
In any case we have a contradiction.
\end{proof}

As a consequence, of the above proposition, if $\rMod{R}$ has only, up to isomorphism, one simple right $R$--module, for instance if either $R$ has only one maximal right ideal, i.e., $R$ is a local ring, then we have the following proposition; compare with  \cite{Barnard:1981b}.

\begin{proposition}\label{pr:32}
Let $R$ be a ring such that $\rMod{R}$ has, up to isomorphism, only a simple right module, for any proper submodule $N\subsetneqq{M}$ the following statements are equivalent:
\begin{enumerate}[(a)]\sepa
\item
$N\subseteq{M}$ is distributive.
\item
$N\subseteq{mR}$ for any $m\in{M\setminus{N}}$.
\item
$N$ is comparable with every non--zero submodule of $M$.
\end{enumerate}
In particular, if $0\neq{N}\subsetneqq{M}$ is a distributive submodule, then $\Soc(M)\subseteq{N}$ and it is essential in $M$.
\end{proposition}
\begin{proof}
(a) $\Rightarrow$ (b). %
If $N\subseteq{M}$ is distributive and $m\in{M}\setminus{N}$ then $\dfrac{N}{N\cap{mR}}$ and $\dfrac{mR}{N\cap{mR}}$ have no simple isomorphic subfactors, hence one of them is equal to zero. If $\dfrac{mR}{N\cap{mR}}=0$, then $mR\subseteq{N}$, which is a contradiction, hence $\dfrac{N}{N\cap{mR}}=0$, and $N\subseteq{mR}$.
\par
(b) $\Rightarrow$ (c). %
Let $H\subseteq{M}$ be a submodule. If $H\nsubseteq{N}$, there exists $h\in{H}\setminus{N}$, hence $N\subseteq{hR}\subseteq{H}$.
\par
(c) $\Rightarrow$ (a). %
Let $H\subseteq{M}$ be a submodule, then either $N\subseteq{H}$, hence $\dfrac{N}{N\cap{H}}=0$, or $H\subseteq{N}$, hence $\dfrac{H}{N\cap{H}}=0$.
\par
By (b) we have that $N\subseteq{M}$ is essential. If $H\subseteq{M}$ is simple and $H\nsubseteq{N}$, there exists $h\in{H\setminus{N}}$, and $N\subseteq{hR}\subseteq{H}$, so $N=H$, which is a contradiction. As a consequence, for any simple submodule $H\subseteq{M}$ we have $H\subseteq{N}$, and $\Soc(M)\subseteq{N}$.
\end{proof}

A second consequence of the afore--mentioned characterization of distributive submodules given in Proposition~\eqref{pr:42} is the following proposition.

\begin{proposition}
Let $M$ be a right $R$--module, $N\subseteq{M}$ a distributive submodule and $H\subseteq{M}$ a submodule such that $N\cap{H}=0$, then $\Hom_R(N,H)=0=\Hom_R(H,N)$.
\end{proposition}
\begin{proof}
For any homomorphism $f:N\longrightarrow{H}$ we have $\Ima(f)$ is a common subfactor of $N$ and $H$, hence $\Ima(f)=0$, and $f=0$. The same happens for any homomorphism $g:H\longrightarrow{N}$.
\end{proof}

Finally, we observe that distributive submodules are preserved by some module constructions.

\begin{proposition}
Let $M$ be a right $R$--module, the following statements hold:
\begin{enumerate}[(1)]\sepa
\item
If $N\subseteq{M}$ is a distributive submodule, for any submodule $H\subseteq{M}$ the submodule $(N+H)/H\subseteq{M/H}$ is distributive.
\item
For every family of distributive submodules $\{N_i\subseteq{M}\mid\;i\in{I}\}$ the sum $\sum_iN_i\subseteq{M}$ is a distributive submodule.
\item
If $N_1,N_2\subseteq{M}$ are distributive submodules, then $N_1\cap{N_2}\subseteq{M}$ is distributive.
\end{enumerate}
\end{proposition}
\begin{proof}
(1). %
Since $N\subseteq{M}$ is distributive, for any $m\in{M}$ and any $n\in{N}$ we have $(N:m)+(mR:n)=R$. The result follows from the following inclusions
\[
(N:m)\subseteq\left(\frac{N+H}{H}:\overline{m}\right)\mbox{ and }
(mR:n)\subseteq(\overline{m}R:\overline{n}),
\]
where $\overline{x}=x+H$ for any $x\in{M}$.
\par
(2) and (3). %
They are well known for finite join and meet of distributive elements of a lattice. It is not difficult to see that in the case of sum it can be extended to the infinite case.
\end{proof}

If $A$ is a commutative ring and $\Sigma\subseteq{A}$ a multiplicatively closed subset, then we have:

\begin{proposition}
Let $M$ be an $A$--module; if $N\subseteq{M}$ is distributive, then $\Sigma^{-1}N\subseteq\Sigma^{-1}M$ is distributive.
\end{proposition}
\begin{proof}
We apply Proposition~(\ref{pr:42}(b)). Let $\frac{m}{1}\in\Sigma^{-1}M$, and $\frac{n}{1}\in\Sigma^{-1}N$. We have the equalities:
\[
\begin{array}{ll}
(\Sigma^{-1}N:\frac{m}{1})=\{\frac{a}{s}\in\Sigma^{-1}A\mid\;\textrm{exists  }t\in\Sigma,\textrm{ such that }mat\in{N}\}=\Sigma^{-1}(N:m),\textrm{ and}\\\\
(\frac{m}{1}\Sigma^{-1}A:\frac{n}{1})=\Sigma^{-1}(mA:n).
\end{array}
\]
Since $(N:m)+(mA:n)=A$, then $(\Sigma^{-1}N:\frac{m}{1})+(\frac{m}{1}\Sigma^{-1}A:\frac{n}{1})=\Sigma^{-1}A$, and $\Sigma^{-1}N\subseteq\Sigma^{-1}M$ is distributive.
\end{proof}

In particular, if $\ideal{p}\subseteq{A}$ is a prime ideal, and consider $\Sigma=A\setminus\ideal{p}$, then we have:

\begin{corollary}
Let $M$ be an $A$--module and $\ideal{p}\subseteq{A}$ be a prime ideal. If $N\subseteq{M}$ is distributive, then $N_\ideal{p}\subseteq{M_\ideal{p}}$ is distributive.
\end{corollary}

Let us consider the following example.

\begin{example}
Let $M=\mathbb{Z}_8$ the cyclic abelian group of eight elements. Since $\mathcal{L}(\mathbb{Z}_8)$ is a distributive lattice, then every submodule is distributive, hence $\Soc(\mathbb{Z}_8)=4\mathbb{Z}_8\subsetneqq2\mathbb{Z}_8$ are proper distributive submodules. Otherwise, $\mathbb{Z}_8$ has no nontrivial direct summands, hence $\mathcal{L}(\mathbb{Z}_8)$ has no a lattice decomposition, see next section.
\end{example}

This means that the existence of distributive submodules does not imply a lattice decomposition. On the other hand, for every lattice decomposition $M=M_1\oplus{M_2}$ we shall prove that $M_1$ and $M_2$ are distributive submodules.

\section{Complemented distributive submodules}\label{se:5}

Let $M=N\oplus{H}$ be a decomposition, and let us denote $i_1:N\longrightarrow{M}$ and $i_2:H\longrightarrow{M}$ the inclusions and $q_1:M\longrightarrow{N}$, $q_2:M\longrightarrow{H}$ the projections. If $p_1:M\longrightarrow{M/N}$ is the projection, there is an isomorphism $f:M/N\cong{H}$ such that $fp_1=q_2$; thus, we have a diagram involving the lattices:
\[
\xy
 (0,0)*+{\mathcal{L}(N)}="a",
 (30,0)*+{\mathcal{L}(M)}="b",
 (60,0)*+{\mathcal{L}(H)}="c",
 \ar@{->}"a";"b"^{i_{1*}}
 \ar@/^+5ex/"b";"c"^{i_2^*}
 \ar@{->}"c";"b"_{i_{2*}}
 \ar@/^+5ex/"b";"a"_{i_1^*}
\endxy
\]
Being $i_{1*}$ and $i_{2*}$ lattice homomorphisms. On the other hand, $i_1^*$ and $i_2^*$ are \textbf{$\wedge$--homomor\-phisms}, i.e., $i_j^*(Y_1\wedge{Y_2})=i_j^*(Y_1)\wedge{i_j^*(Y_2)}$, for every $Y_1,Y_2\subseteq{M}$, $j=1,2$, and they are lattice homomorphisms whenever $N$, or equivalently $H$, is a distributive submodule. In this case, $\mathcal{L}(M)$ is the direct product of $\mathcal{L}(N)$ and $\mathcal{L}(H)$; and for every $Y\subseteq{M}$ we have $Y=i_{1*}i_1^*(Y)\vee{i_{2*}i_2^*(Y)}$.

The first result is a direct consequence of the characterizations of distributive submodules, in Proposition~\eqref{pr:42}, which are direct summands, i.e., \textbf{complemented distributive submodules}.

\begin{lemma}\label{le:1510061641}
Let $M=N\oplus{H}$ be a direct sum, the following statements are equivalent:
\begin{enumerate}[(a)]\sepa
\item
$N\subseteq{M}$ is distributive.
\item
$N$ and $H$ have no isomorphic simple subfactors.
\item
$\Ann(n)+\Ann(h)=R$ for any $n\in{N}$ and $h\in{H}$.
\item
For any submodule $X\subseteq{M}$ we have $X=(X\cap{N})+(X\cap{H})=(X+N)\cap(X+H)$.
\item
$H\subseteq{M}$ is distributive.
\end{enumerate}
\end{lemma}

If $M$ is a right $R$--module satisfying the equivalent statements in the above lemma we say that $M=N\oplus{H}$ is a \textbf{lattice decomposition} of $M$.

In this context, if $A$ is a commutative ring, we have the following result that characterizes complemented distributive submodules.

\begin{corollary}\label{co:20200806}
Let $A$ be a commutative ring, $M$ an $A$--module, and $M=N\oplus{H}$ be a decomposition in a direct sum, the following statements are equivalent:
\begin{enumerate}[(a)]\sepa
\item
$N\subseteq{M}$ is a distributive submodule.
\item
$\Supp(N)\cap\Supp(H)=\varnothing$.
\item
$H\subseteq{M}$ is a distributive submodule.
\end{enumerate}
\end{corollary}
\begin{proof}
(a) $\Rightarrow$ (b). %
By Lemma~\eqref{le:1510061641}, for any $n\in{N}$ and any $h\in{H}$ we have $\Ann(n)+\Ann(h)=A$. If $\ideal{p}\in\Supp(N)\cap\Supp(H)$, there exist $n\in{N}$ and $h\in{H}$ such that $(Rn)_\ideal{p}\neq0$, and $(Rh)_\ideal{p}\neq0$, hence $\Ann(n)\subseteq\ideal{p}$ and $\Ann(h)\subseteq\ideal{p}$, which is a contradiction.
\par
(b) $\Rightarrow$ (a). %
If $N$ is not distributive, by Lemma~\eqref{le:1510061641}, $N$ and $H$ have isomorphic simple subfactors, hence $\Supp(N)\cap\Supp(H)\neq\varnothing$, which is a contradiction.
\end{proof}

As a consequence, if in addition $A$ is a noetherian ring, then lattice decomposition is inherited by injective hulls.

\begin{corollary}\label{co:20200806b}
Let $A$ be a commutative noetherian ring, $M$ an $A$--module, and $M=N\oplus{H}$ be a lattice decomposition, then $E(M)=E(N)\oplus{E(H)}$ is a lattice decomposition.
\end{corollary}
\begin{proof}
It is a direct consequence of the well known fact that $\Supp(N)=\Supp(E(N))$ for any $A$--module $N$.
\end{proof}

Also we have the following straightforward result.

\begin{lemma}
Let $M$ be a right $R$--module such that $\mathcal{L}(M)$ is a direct product of two lattices, say $\mathcal{L}(M)=\mathcal{L}_1\times\mathcal{L}_2$, there exist $M_1,M_2\subseteq{M}$ such that
\begin{enumerate}[(1)]\sepa
\item
$M=M_1\oplus{M_2}$.
\item
$M_1$ and $M_2$ are distributive submodules.
\item
$\mathcal{L}_i\cong\mathcal{L}(M_i)$, for every $i=1,2$.
\item
$\mathcal{L}(M)=[0,M_1]\times[0,M_2]$.
\item
There exists an idempotent endomorphism $e\in\End_R(M)$ such that $e(M)=M_1$, and $(1-e)(M)=M_2$. In addition, $e_{\mid{M_1}}=\id_{M_1}$, and $(1-e)_{\mid{M_2}}=\id_{M_2}$.
\end{enumerate}
\end{lemma}

The existence of a non trivial idempotent endomorphism in $\End_R(M)$ is necessary, but it is not sufficient to get a lattice decomposition, i. e., not every idempotent endomorphism $e\in\End_R(M)$ defines a lattice  decomposition of $\mathcal{L}(M)$. Let us illustrate it by some examples.

\begin{example}
Let us consider the abelian group $M=\mathbb{Z}_2\times\mathbb{Z}_2$. It is clear that $M$ has not non trivial distributive submodules, but $\End(M)$ has non trivial idempotents.
Indeed, the ring $\End(\mathbb{Z}_2\times\mathbb{Z}_2)=M_2(\mathbb{Z}_2)$ has six non trivial idempotent endomorphisms
$\begin{pmatrix}1&1\\0&0\end{pmatrix}$,
$\begin{pmatrix}1&0\\1&0\end{pmatrix}$,
$\begin{pmatrix}0&0\\1&1\end{pmatrix}$,
$\begin{pmatrix}0&1\\0&1\end{pmatrix}$,
$\begin{pmatrix}1&0\\0&0\end{pmatrix}$ and
$\begin{pmatrix}0&0\\0&1\end{pmatrix}$,
but no one of them defines a lattice decomposition.
\end{example}

\begin{example}
In the positive we have: %
If we consider the abelian group $M=\mathbb{Z}_2\times\mathbb{Z}_3$, then $\End(\mathbb{Z}_2\times\mathbb{Z}_3)\cong\End(\mathbb{Z}_2)\times\End(\mathbb{Z}_3)$; this ring decomposes, and there is a non trivial idempotent that produces a lattice decomposition of $M$.
\end{example}

The next example shows that in a lattice decomposable module not every idempotent endomorphism provides a lattice decomposition.

\begin{example}\label{ex:150916}
Let us consider the abelian group $M=\mathbb{Z}_2\times\mathbb{Z}_2\times\mathbb{Z}_3$. The lattice of all subgroups is:
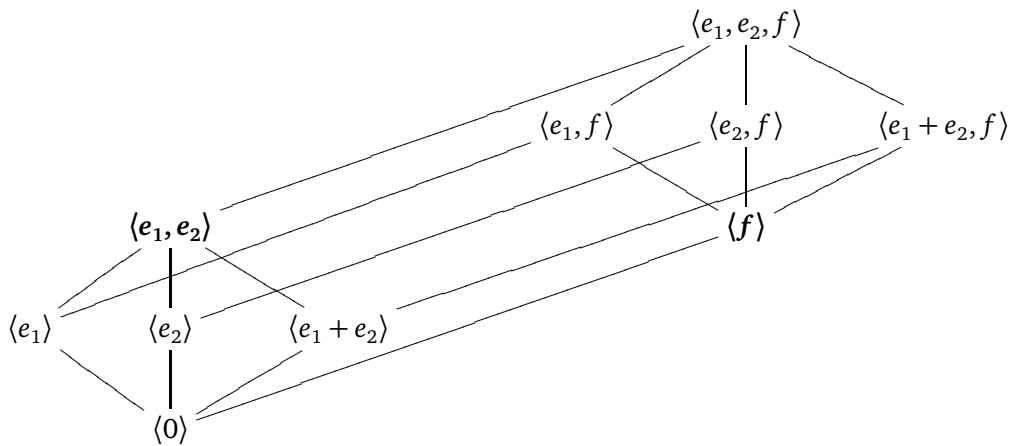
\[
\resizebox{13.5cm}{!}{
\begin{xy}
\xymatrix{
&&&&&\langle{e_1,e_2,f}\rangle\ar@{-}[ld]\ar@{-}[d]\ar@{-}[rd]\\
&&&&\langle{e_1,f}\rangle&\langle{e_2,f}\rangle&\langle{e_1+e_2,f}\rangle\\
&{\boldsymbol{\langle{e_1,e_2}\rangle}}
    \ar@{-}[ld]\ar@{-}[d]\ar@{-}[rd]\ar@{-}[rrrruu]&&&&
    {\boldsymbol{\langle{f}\rangle}}\ar@{-}[lu]\ar@{-}[u]\ar@{-}[ru]\\
\langle{e_1}\rangle\ar@{-}[rrrruu]&\langle{e_2}\rangle\ar@{-}[rrrruu]&
    \langle{e_1+e_2}\rangle\ar@{-}[rrrruu]&&\\
&\langle{0}\rangle\ar@{-}[lu]\ar@{-}[u]\ar@{-}[ru]\ar@{-}[rrrruu]&&&&\\
}\end{xy}}
\]
The decomposition given by $N_1={\boldsymbol{\langle{e_1,e_2}\rangle}}$, $N_2={\boldsymbol{\langle{f}\rangle}}$ corresponds to the idempotent central endomorphism $e\in\End(M)$ defined by
\[
e\left\{\begin{array}{lcl}
(e_1)&=&e_1\\
(e_2)&=&e_2\\
(f)&=&0
\end{array}\right.
\quad
1-e\left\{\begin{array}{lcl}
(e_1)&=&0\\
(e_2)&=&0\\
(f)&=&f
\end{array}\right.
\]
and it defines the lattice decomposition of $\mathcal{L}(M)$ represented in the above diagram. In addition, $e$ defines a lattice decomposition of the ring $S=\End_R(M)$.
\end{example}

\begin{question}
Does every central idempotent endomorphism in $\End_R(M)$ induce a lattice decomposition of $\mathcal{L}(M)$?
\end{question}

The answer is no, as the following example shows.

\begin{example}\label{ex:8}
Let us consider $M=\mathbb{Z}_{(2)}\times\mathbb{Z}_{(3)}$, where $\mathbb{Z}_{(2)}$ and $\mathbb{Z}_{(3)}$ are the localization of $\mathbb{Z}$ at $2\mathbb{Z}$ and $3\mathbb{Z}$, respectively. We claim $\Hom(\mathbb{Z}_{(2)},\mathbb{Z}_{(3)})=0$. Indeed, for any $f\in\Hom(\mathbb{Z}_{(2)},\mathbb{Z}_{(3)})$, let $f(1)=\frac{a}{d}$, then $f(\frac{1}{3})=\frac{b}{c}$ and it satisfies: $3\frac{b}{c}=\frac{a}{d}$, and $3bd=ac$. By hypothesis $3\nmid{c}$, hence $3\mid{a}$. Similarly, if $f(\frac{1}{3^t})=\frac{b}{c}$, then $3^tdb=ac$, and $3^t\mid{a}$ for every $t\in\mathbb{N}$, which implies $a=0$.

It is clear that $\End(M)=\End(\mathbb{Z}_{(2)})\times\End(\mathbb{Z}_{(3)})$; hence, in $\End(M)$ there exist central idempotent elements. Since $\mathbb{Z}_{(2)}$ and $\mathbb{Z}_{(3)}$ have a non zero isomorphic submodule in common, then $\mathbb{Z}_{(2)}\subseteq{M}$ is not distributive.
\end{example}

\begin{remark}
Observe that any central idempotent element $e\in\End_R(M)=S$ induces a complemented distributive two--sided ideal $eS\subseteq{S}$, with complement $(1-e)S$, see Corollary~\eqref{co:20200804} below. The above example shows that if $e\in\End_R(M)$ is a central idempotent element, then the submodule $e(M)\subseteq{M}$ is not necessarily distributive, even if $S$ has a lattice decomposition.
\end{remark}

In conclusion, the question is: how we may describe the idempotent elements in $\End_R(M)$ that produce lattice decomposition? The next theorem provides the answer.

\begin{theorem}\label{th:151011}
Let $M$ be a right $R$--module, $N\subseteq{M}$ be a direct summand, and let $e\in\End_R(M)$ be an idempotent such that $e(M)=N$, the following statements are equivalent:
\begin{enumerate}[(a)]\sepa
\item
$N$ is complemented distributive.
\item
$e$ is a central and $e(X)\subseteq{X}$ for any submodule $X\subseteq{M}$.
\end{enumerate}
\end{theorem}
\begin{proof}
If $N\subseteq{M}$ is a complemented distributive submodule with complement $H$, then for every submodule $X\subseteq{M}$ we have $X=(N\cap{X})+(H\cap{X})=(e(M)\cap{X})+((1-e)(M)\cap{X})$, that expressed in terms of the endomorphism $e$, implies:
$$
e(X)=e(e(M)\cap{X})+e((1-e)(M)\cap{X})\subseteq{e(M)\cap{X}}\subseteq{X},
$$
Hence a necessary condition on the endomorphism $e$ to get a complemented distributive submodule is $e(X)\subseteq{X}$, for any submodule $X\subseteq{M}$. This is also a sufficient condition; indeed, if $e(X)\subseteq{X}$ (or equivalently $(1-e)(X)\subseteq{X}$), then, for any element $x\in{X}$ we have $x=e(x)+(1-e)(x)$, where $e(x)\in{e(M)\cap{X}}$ and $(1-e)(x)\in{(1-e)(M)\cap{X}}$.
\par
Let $M$ be a right $R$--module $M$, for any complemented distributive submodule $N\subseteq{M}$ with complement $H$, following Cohn's theory in \cite{COHN:1985}, for any homomorphism $f:N\longrightarrow{H}$ we define a submodule $\Gamma(f)=\{(x,f(x))\in{M}\mid\;x\in{N}\}$. Since $\Gamma(f)$ is a complement of $N$, and it is unique, it follows that $\Gamma(f)=\Gamma(0)=N$. In particular, $\End_R(N,H)=0$, hence we have $\End_R(M)\cong\End_R(N)\times\End_R(H)$.
As a consequence $\End_R(N)$ is a direct summand ideal of $\End_R(M)$, and there exists a central idempotent $e\in\End_R(M)$ such that $\End_R(N)=e\End_R(M)$. Therefore, the idempotent endomorphism that defines the lattice decomposition is central.
\end{proof}

An idempotent endomorphism $e\in\End_R(M)$ is named \textbf{fully invariant} if $e(X)\subseteq{X}$ for any submodule $X\subseteq{M}$.

\begin{corollary}\label{co:20200804}\label{co:20200910}
Let $R$ be a ring, and $e\in{R}$ idempotent, the following statements are equivalent:
\begin{enumerate}[(a)]\sepa
\item
$eR\subseteq{R}$ is a complemented distributive right ideal.
\item
$e$ is central.
\item
$R=Re\times{R(1-e)}$ is a direct product of rings.
\end{enumerate}
In this case $R$ has a lattice decomposition in a direct product of two (twosided) ideals: $Re$ and $R(1-e)$, and conversely; any direct sum decomposition $R=\Ideal{a}\oplus\Ideal{b}$, with two (two--sided) ideals, is a lattice decomposition.
\end{corollary}

\begin{remark}
Observe that if $N\subseteq{M}$ is a complemented distributive submodule, defined by the idempotent endomorphism $e\in\End_R(M)$, then $\End_R(M)e\subseteq\End_R(M)$ is a complemented distributive two--sided ideal; therefore a ring with unity $e$.
\end{remark}

In particular, we have the next result that enhances \cite[Proposition~4.3]{Stephenson:1974} and \cite[Proposition~1.3]{Barnard:1981b}.

\begin{lemma}
Every complemented distributive submodule $N\subseteq{M}$ is stable under any $f\in\End_R(M)$.
\end{lemma}
\begin{proof}
By hypothesis there exists $e\in\End_R(M)$, central idempotent, such that $N=e(M)$. For every endomorphism $f\in\End_R(M)$ we have $f(N)=fe(M)=ef(M)\subseteq{e(M)}=N$.
\end{proof}

If we analysed the question if every central idempotent endomorphism $e\in\End_R(M)$ defines a complemented distributive submodule $e(M)\subseteq{M}$. We found that Example~\eqref{ex:8} gave a negative answer. A necessary and sufficient condition in order to get $e(M)\subseteq{M}$ a complemented distributive submodule, as we have seen before, is that for every submodule $X\subseteq{M}$ we have $e(X)\subseteq{X}$. In consequence, for every element $m\in{M}$ we have $e\langle{m}\rangle\subseteq\langle{m}\rangle$, and there exists $a_m\in{R}$ such that $em=ma_m$.

Let us explore this condition in order to get more examples of complemented distributive submodules.

Of particular interest is the situation in which $A$ is a commutative ring. In this case we have:

\begin{proposition}\label{pr:151011}
Let $A$ be a commutative ring and $N\subseteq{M}$ be a submodule of the $A$--module $M$, the following statements are equivalent:
\begin{enumerate}[(a)]\sepa
\item\label{it:pr:151011-1}
$N\subseteq{M}$ is a complemented distributive submodule.
\item\label{it:pr:151011-2}
There exists a central idempotent endomorphism $e\in\End_A(M)$ such that $e(M)=N$ and $e$ belongs to the closure of $A/\Ann(M)\subseteq\End_A(M)$ in the finite topology, i.e., the topology with subbase of neighbourhoods of any $f\in\End_A(M)$ given by $B(\{m_1,\ldots,m_t\},f)=\{g\in\End_A(M)\mid\;g(m_i)=f(m_i),\,i=1,\ldots,t\}$.
\end{enumerate}
\end{proposition}
\begin{proof}
\eqref{it:pr:151011-1} $\Rightarrow$ \eqref{it:pr:151011-2}. %
Let $N\subseteq{M}$ be a complemented distributive submodule and $e\in\End_A(M)$ be a central idempotent element such that $e(M)=N$, and $e(X)\subseteq{X}$ for every submodule $X\subseteq{M}$. If $H=(1-e)(M)$ for any $m\in{M}$ there are $n\in{N}$ and $h\in{H}$ such that $m=n+h$, and there exists $a\in{A}$ such that $n=e(m)=ma=(n+h)a$, hence $ha=0$ and $n(1-a)=0$.
\par
Let $m_1,\ldots,m_t\in{M}$ with $m_i=n_i+h_i$ and $e(m_i)=m_ia_i$. Let us consider new elements:
$m_{ij}=n_i+h_j$; observe that there are elements $a_{ij}\in{A}$ such that $e(m_{ij})=m_{ij}a_{ij}$ and $a_{ii}=a_i$; therefore $n_{ij}(1-a_{ij})=0$ and $h_{ij}a_{ij}=0$, for any $i,j\in\{1,\ldots,t\}$.
\par
If we fix $i$, we have the elements $m_{i1}=n_i+h_1,\ldots,m_{it}=n_i+h_t$, that satisfy:
\[
\begin{array}{lll}
n_i(1-a_{i1}\cdots{a_{it}})=n_i-n_i(a_{i1}\cdots{a_{it}})=0,\\
h_j(a_{i1}\cdots{a_{it}})=0,\\
e(m_{ij})=m_{ij}(a_{i1}\cdots{a_{it}})
\end{array}
\]
Thus we may assume $a_{i1}=\cdots=a_{it}$, and all of them are equal to the product $a_{i1}\cdots{a_{it}}j$ for the former $a_{ij}$. Since this can be done for every index $i$, then we may assume $a_{i1}=\cdots=a_{it}$ for every index $i$.
\par
In this new context, if we fix $j$, we have elements $m_{1j}=n_1+h_j,\ldots,m_{tj}=n_t+h_j$, that satisfy:
\[
\begin{array}{lll}
n_i(1-a_{1j})\cdots(1-a_{tj})=0, \textrm{ for every index }i,\\
(1-a_{1j})\cdots(1-a_{tj})=1-\sum_ia_{ij}+\sum_{i_1<i_2}a_{i_1j}a_{i_2j}
    +\cdots+(-1)^t{a_{1j}}\cdots{a_{tj}},\\
\textrm{let us denote } x=\sum_ia_{ij}-\sum_{i_1<i_2}a_{i_1j}a_{i_2j}+\cdots+(-1)^{t+1}{a_{1j}}\cdots{a_{tj}},\\
h_jx=0,\\
e(m_{ij})=m_{ij}x.
\end{array}
\]
Thus we may assume $a_{1j}=\cdots=a_{tj}$, and all of them are equal to the element $x$ defined just before using the former $a_{ij}$. Since this can be done for every index $j$, then we may assume $a_{1j}=\cdots=a_{tj}=x$ for every index $j$. In consequence, we have found that $e(m_{ij})=m_{ij}x$, and in particular $x\in{B(\{m_1,\ldots,m_t\},e)}\cap{A}$. Therefore, $e$ belongs to the closure of $A/\Ann(M)$ in the finite topology of $\End_A(M)$.
\par
\eqref{it:pr:151011-2} $\Rightarrow$ \eqref{it:pr:151011-1}. %
It is consequence of Theorem~\eqref{th:151011}.
\end{proof}

\begin{corollary}
For any commutative ring $A$ and any finitely generated $A$--module $M$, if $N\subseteq{M}$ is a complemented distributive submodule with idempotent endomorphism $e$, there exists $a\in{A}$ such that $e(m)=ma$ for any $m\in{M}$.
\end{corollary}

Even in this particular case we may characterize complemented distributive submodules.

\begin{corollary}
Let $A$ be a commutative ring and let $M$ be a finitely generated $A$--module; every complemented distributive submodule $N\subseteq{M}$ determines an idempotent element in $A/\Ann(M)$. And conversely, every non zero idempotent in $A/\Ann(M)$ defines a nonzero complemented and distributive submodule of $M$.
\end{corollary}

\begin{example}
Let us consider the abelian group $M=\mathbb{Z}_2\times\mathbb{Z}_2\times{Z}_3$ in the Example~\eqref{ex:150916}; we have $\Ann(M)=6\mathbb{Z}$, and the non trivial idempotents of $\mathbb{Z}/6\mathbb{Z}$ are $\overline{4}$ and $\overline{3}$, which define the subgroups $2M=\langle{f}\rangle$ and $3M=\langle{e_1,e_2}\rangle$.
\end{example}

\section{Decomposition of categories. Examples}\label{se:06}

In this section we will apply the above result to study the decomposition of the category of right $R$--modules and the category $\sigma[M]$ defined by a right $R$--module $M$.

\subsection*{The category $\rMod{R}$}

Each lattice decomposition of $R$, as right $R$--module, is defined by a central idempotent element $e\in{R}$, hence $R=Re\times{R(1-e)}$, being $Re$ and $R(1-e)$ rings (and (twosided) ideals). Therefore, we have a decomposition of the module category as $\rMod{R}\cong\rMod{eR}\times\rMod{(1-e)R}$.

\subsection*{The category $\sigma[M]$}

The lattice decomposition theory of a right $R$--module $M$ is closely linked to the structure of the $M$ module and also to the structure of submodules of the modules it generates; there is a category that studies precisely these modules: the category $\sigma[M]$. The category $\sigma[M]$, as defined in \cite{WISBAUER:1988}, is the full subcategory of $\rMod{R}$ whose objects are all the right $R$--modules isomorphic to modules subgenerated by $M$, i.e., submodules of factors of direct sums of copies of $M$.

Our aim is to study under which circumstances the category $\sigma[M]$ is a direct product of two categories $\mathcal{N}$ and $\mathcal{H}$.

Let us assume $F:\sigma[M]\cong\mathcal{N}\times\mathcal{H}$, without losing of generality we may assume $\mathcal{N}$ and $\mathcal{H}$ are subcategories of $\sigma[M]$, the objects of $\mathcal{N}\times\mathcal{H}$ are pairs $(X,Y)$, where $X$ is and object of $\mathcal{N}$ and $Y$ an object of $\mathcal{H}$, and $F(M)=(N,H)$ satisfying $M=N\oplus{H}$.

First we need a technical result, which will be useful in later developments.

\begin{proposition}\label{pr:10}
Let $N\subseteq{M}$ be a complemented distributive submodule, for every index set $I$ we have $N^{(I)}\subseteq{M^{(I)}}$ is a complemented distributive submodule. The reciprocal also holds.
\end{proposition}
\begin{proof}
It is obvious that if $M=N\oplus{H}$, then $M^{(I)}=N^{(I)}\oplus{H^{(I)}}$. Otherwise, if $S$ is a simple subfactor of $N^{(I)}$, there exists a finite subset $F\subseteq{I}$ such that $S$ is a subfactor of $N^{(F)}$, hence $S$ is a subfactor of $N$. Since $N$ and $H$ have no isomorphic simple subfactors, $N^{(I)}$ y $H^{(I)}$ have no isomorphic simple subfactors, hence $N^{(I)}\subseteq{M^{(I)}}$ is a complemented distributive submodule.
\end{proof}

Every submodule $Z\subseteq{M}$ corresponds to a pair $(X,Y)$, object of $\mathcal{N}\times\mathcal{H}$ satisfying $X\subseteq{N}$, $Y\subseteq{H}$ and $Z=X\oplus{Y}$, hence $N\subseteq{M}$ is a distributive submodule of $M$ with complement $H$.

Conversely, for any complemented and distributive submodule $N\subseteq{M}$, with complement $H\subseteq{M}$, and any index set $I$, we have $N^{(I)}\subseteq{M^{(I)}}$ is a complemented and distributive submodule with complement $H^{(I)}$, see Proposition~\eqref{pr:10}, and any submodule $X\subseteq{M^{(I)}}$ can be written as $X=(X\cap{N^{(I)}})\oplus(X\cap{H^{(I)}})$, hence  $\dfrac{M^{(I)}}{X}\cong\dfrac{N^{(I)}}{X\cap{N^{(I)}}}\oplus\dfrac{H^{(I)}}{X\cap{H^{(I)}}}
\cong\dfrac{N^{(I)}+X}{X}\oplus\dfrac{H^{(I)}+X}{X}$, and $\dfrac{N^{(I)}}{X\cap{N^{(I)}}}\cong\dfrac{N^{(I)}+X}{X}\subseteq\dfrac{M}{X}$ is a complemented and distributive submodule. As a consequence, every submodule $Y\subseteq\dfrac{M^{(I)}}{X}$ can be written as $Y=\left(Y\cap\dfrac{N^{(I)}+X}{X}\right)\oplus\left(Y\cap\dfrac{N^{(I)}+X}{X}\right)$, where $Y\cap\dfrac{N^{(I)}+X}{X}$ is an object of $\sigma[N]$, and $Y\cap\dfrac{H^{(I)}+X}{X}$ is an object of $\sigma[H]$. Therefore there is a category isomorphism $\sigma[M]\cong\sigma[N]\times\sigma[H]$. Compare with \cite[Proposition 2.2]{Vanaja:1996} and \cite[2.4]{Wisbauer:2000}.

\begin{theorem}\label{th:61}
With the above assumptions. Let $M$ be a right $R$--module, the following statements are equivalent:
\begin{enumerate}[(a)]\sepa
\item
$\sigma[M]$ is a direct product of two categories, $\sigma[M]\cong\sigma[N]\times\sigma[H]$.
\item
$M$ has a complemented and distributive submodule $N\subseteq{M}$, with complement $H$.
\end{enumerate}
\end{theorem}

As a consequence of Corollary~\eqref{co:20200806b}, if $A$ is a noetherian commutative ring, $M$ an $A$--module, and $E(M)$ its injective hull, for any complemented distributive submodule $N\subseteq{M}$ we have that $E(N)\subseteq{E(M)}$ is distributive. This result does not necessarily hold in a non--commutative framework as the following example shows, see \cite{Garcia/Jara/Merino:1999}.

\begin{example}
Let $K$ be a field and $R=\begin{pmatrix}K&0\\K&K\end{pmatrix}$ be a ring. The maximal right ideals of $R$ are
$\Ideal{p}=\begin{pmatrix}K&0\\K&0\end{pmatrix}$ and
$\Ideal{q}=\begin{pmatrix}0&0\\K&K\end{pmatrix}$; hence there are, up to isomorphism, two different simple right $R$--modules. Let us consider the right $R$--module
$N=\begin{pmatrix}K&0\\0&0\end{pmatrix}$, which is isomorphic to $R/\Ideal{q}$, and the cyclic right $R$--module
$E=\begin{pmatrix}K&K\\0&0\end{pmatrix}$, generated by $\begin{pmatrix}0&1\\0&0\end{pmatrix}$. Since $E$ is injective and the inclusion $N\subseteq{E}$ is essential, then $E$ is the injective hull of $N$. In addition we have an isomorphism $E/N\cong{R/\Ideal{p}}$.
\par
Consider now the right $R$--module $M=N\oplus(E/N)$. Since both factors are simple right $R$--module, $M$ has a lattice decomposition; $N\subseteq{M}$ is a complemented distributive submodule. Otherwise, $E(M)=E(N)\oplus{E(E/N)}=E\oplus(E/N)$, and $E\subseteq{E(M)}$ is not a distributive submodule. Indeed, $E(M)$ has no nontrivial complemented distributive submodules.
\end{example}

\end{document}